\documentclass[11pt,leqno]{amsart}
\usepackage[utf8]{inputenc}
\usepackage{amsfonts,amssymb}
\usepackage{graphicx,tikz}
\usepackage{tikz-cd} 
\usepackage{float}
\usepackage{centernot}
\usepackage{amsmath, amsthm, amsfonts}
\usepackage{geometry}
\usepackage{hyperref}
\usepackage[capitalise]{cleveref}
\usepackage{enumerate}
\usepackage{enumitem}
\usepackage{amsrefs}
\usepackage{nicefrac}
\usepackage[official]{eurosym}
\usepackage{graphicx}
\usepackage{nomencl}
\usepackage{amsfonts}
\usepackage{hyperref}
\usepackage{amssymb}
\usepackage{fancyhdr}
\usepackage{amscd}
\usepackage[english]{babel}

\newcommand{\Z}{\mathbb{Z}}

\newcommand{\Q}{\mathbb{Q}}

\theoremstyle{plain}
\theoremstyle{definition}

\newtheorem*{theorem*}{Theorem}
\newtheorem{thm}{Theorem}[section]

\newtheorem{lem}[thm]{Lemma}

\theoremstyle{definition}

\newtheorem{example}[thm]{\scshape{Example}}

\newtheorem{rem}[thm]{\scshape{Remark}}

\newcommand{\doesnotdivide}{\not\hspace{2.5pt}\mid}
\makeatletter
\def\cleardoublepage{\clearpage\if@twoside \ifodd\c@page\else
	\hbox{}
	\thispagestyle{empty}
	\newpage
	\if@twocolumn\hbox{}\newpage\fi\fi\fi}
\makeatother

\title{Powerful 3-Engel groups}
\author[I. de las Heras]{Iker de las Heras}
\address{Iker de las Heras: Mathematisches Institut der Heinrich-Heine-Universität Düsseldorf
Universitätsstr. 1, D-40225 Düsseldorf, Germany}
\email{iker.delasheras@ehu.eus}
\author[M. Noce]{Marialaura Noce}
\address{Marialaura Noce: Department of Mathematics, University of Salerno, via Giovanni Paolo II 132, 84084 Fisciano (SA), Italy}
\email{mnoce@unisa.it}
\author[G. Traustason]{Gunnar Traustason}
\address{Department of Mathematical Sciences, University of Bath, Claverton Down, Bath BA2 7AY, United Kingdom}
\email{gt223@bath.ac.uk}

\keywords{Engel groups, powerful groups, metabelian groups, nilpotency class}
\subjclass[2010]{20F45, 20D15, 20F19}

\usepackage{graphicx}
\begin{document}

\maketitle
\begin{abstract}
In this paper we study powerful 3-Engel groups. In particular, we find sharp upper bounds for the nilpotency class of powerful $3$-Engel groups and the subclass of powerful metabelian $3$-Engel groups.
\end{abstract}

\section{Introduction}
\mbox{}\\
Let $n$ be a positive integer. A group $G$ is said to be an $n$-Engel group if for all $x,y\in G$ we have $[[y,x],\overset{n}{\ldots}\,,x]=1.$
It is a long-standing question whether a finitely generated $n$-Engel group is nilpotent.
By Wilson \cite{Wi91}, a finite $r$-generator $n$-Engel group is nilpotent of $(n,r)$-bounded class and hence any finitely generated residually finite $n$-Engel group is nilpotent.
Whereas the class of a finite $n$-Engel groups depends on the number of generators when $n\geq 3$, the situation is different if we furthermore assume that $G$ is a powerful $p$-group.

Recall that a finite $p$-group $G$ is powerful if $[G,G]\leq G^{p}$, when $p$ is odd, and $[G,G]\leq G^{4}$, when $p=2$. The reason why the nilpotency class of a powerful $n$-Engel group is only bounded in $n$ comes from a well known theorem of Burns and Medvedev \cite{BuMe98}, according to which there exist positive integers $m$ and $k$ only depending on $n$ such that any finite $n$-Engel group satisfies $\gamma_{m}(G)^{k}=1$.
 
Let $G$ be a powerful $m$-Engel $p$-group. It is a well known property of powerful $p$-groups that $\gamma_{m+l}(G)\leq \gamma_{m}(G)^{p^{l}}$. Thus if we take $m=m(n)$ and $k=k(n)$ as in the theorem of Burns and Medvedev, then picking $l$ such that $p^{l}\geq k$ for all $p$ we see that every powerful $n$-Engel $p$-group is nilpotent of class at most $c=m+l-1$.
Notice that the class is bounded only by $n$ and not by $r$ or $p$.

In \cite{PrTr08} and \cite{Tr08} it is shown that the best upper bound for the nilpotency class of powerful $2$-Engel groups is $3$. In this case we get the same upper bound as holds for $2$-Engel groups in general. It is though interesting to note that it is shown that every powerful 
$3$-generator $2$-Engel group is nilpotent of class at most $2$ and thus satisfies the law $[x,y,z]=1$. This may seem counter intuitive but one should remember that the class of powerful $p$-groups is not closed under taking subgroups. \\ \\
In this paper we will turn our focus on powerful $3$-Engel groups. %
The latter and, more precisely, the behaviour of 3-Engel elements in groups has been developed and studied recently \cites{Trac,Trau1,NoTrTr, HadNoTr}.

We now describe our main results. We start with a general result on powerful $n$-Engel groups. Note that every powerful group $G$ of rank $2$ has a cyclic derived subgroup. In such a case, if $G$ is also $n$-Engel, we show that the nilpotency class of $G$ can be at most $n$.

\begin{thm}
 \label{thm: derived cyclic}
 Let $n\ge 1$ and let $G$ be an $n$-Engel group with cyclic derived subgroup.
 Then the nilpotency class of $G$ is at most $n$.
 Moreover, this bound is best possible.
\end{thm}

\begin{rem}
 If $G$ is a $2$-generator non-powerful $3$-Engel $p$-group, then its derived subgroup need not be cyclic.
 In that case the nilpotency class of $G$ is at most $3$ if $p$ is odd and $4$ if $p=2$ (this was shown in \cite{He61} and it is summed up in Theorem 2.1 of \cite{guptanewman}).
 These bounds are also best possible.
\end{rem}
Our main result is determining the best upper bound for the nilpotency class of powerful $3$-Engel groups.
From \cite{guptanewman} we know that every $3$-Engel group satisfies $\gamma_{5}(G)^{20}=1$, and thus $\gamma_5(G)=1$ when $G$ is a $p$-group with $p\neq 2,5$, $\gamma_{5}(G)^{5}=1$ when $p=5$ and $\gamma_{5}(G)^{4}=1$ when $p=2$.
From the general analysis above we thus know that when $G$ is furthermore powerful then $\gamma_{6}(G)=1$.
We will show that $5$ is the best upper bound for the nilpotency class.
The following theorem gives a more detailed analysis. 
%

\begin{thm}
 \label{thm: non-metabelian}
 Let $G$ be a powerful $3$-Engel $p$-group. If $p\neq 2,5$, then $G$ is nilpotent of class at most $4$.
 Furthermore, if $G$ has rank $r\ge 3$, then: 
 \begin{enumerate}
     \item\label{it: p neq2,5,r=3}If $r=3$, then the nilpotency class of $G$ is at most $3$;
     \item If $r\ge 4$, then the nilpotency class of $G$ is at most $4$.
 \end{enumerate}
 If $p=2$ or $5$, then $G$ is nilpotent of class at most $5$.
 Furthermore:
 \begin{enumerate}
    \item[(iii)]\label{it: p=5,r=3} If $r=3$ and $p=5$, then the nilpotency class of $G$ is at most $3$;
    \item[(iv)] \label{it: p=2,r=3}If $r=3$ and $p=2$,  then the nilpotency class of $G$ is at most $4$;
    \item[(v)] If $4\le r\le 5$, then the nilpotency class of $G$ is at most $4$;
  \item[(vi)] If $r\ge 14$, then the nilpotency class of $G$ is at most $5$.
 \end{enumerate}
 Moreover, these upper bounds are best possible.
\end{thm}
\begin{rem} For $6\leq r\le 13$, we only know that the best upper bound for the class is $4$ or $5$.
\end{rem}
Finally, when we furthermore know that $G$ is metabelian then we can get the following result.

\begin{thm}
 \label{thm: metabelian}
 Let $G$ be a powerful metabelian $3$-Engel $p$-group of rank $r\ge 3$.
 If $p$ is odd, then the nilpotency class of $G$ is at most $3$.
 If $p=2$, then the nilpotency class of $G$ is at most $4$.
 Moreover, these upper bounds are best possible.
\end{thm}

\begin{rem}
\label{rem: p odd or r=3}
 In Theorem \ref{thm: metabelian}, if $p$ is odd or if $r=3$, then the assumption that $G$ is powerful is not necessary, as shown in Theorem \ref{thm: gamma4^16=1} and Lemma \ref{lem: metabelian 3-generator}. 
\end{rem}

\noindent\textit{Notation and organisation.}
Our notation is mostly standard.
For a group $G$ and elements $x,y\in G$, the commutator of $x$ and $y$ is the element $[x,y]=x^{-1}y^{-1}xy$.
We  use left-normed notation for commutators of length $\ge 3$ (for example, $[x,y,z] = [[x,y],z]$).
A commutator of length $n\in\mathbb{N}$ in a subset $X\subseteq G$ is a commutator of the form $[x_1,\ldots,x_n]$, where $x_1,\ldots,x_n\in X$.
We denote the derive subgroup of $G$ by $G'$.

The paper is organized as follows.
In Section \ref{sec: preliminaries} we will recall and develop some useful tools about commutators in 3-Engel groups.
In Section \ref{sec: metabelian}, we prove Theorem \ref{thm: metabelian} which gives a bound for the nilpotency class of a powerful metabelian 3-Engel group. In Section \ref{sec: powerful rank 3} we establish Theorem \ref{thm: non-metabelian}(i), (iii) and (iv) for powerful $3$-Engel groups of rank $3$.
Finally, in Section \ref{sec: powerful rank r}, we complete the proof of Theorem \ref{thm: non-metabelian} and, as a consequence, the description of the nilpotency class of powerful 3-Engel groups.

\section{Preliminaries}
\label{sec: preliminaries}

Throughout the paper, the following identities due to Gupta and Newman \cite{guptanewman} will have a prominent role.

\begin{lem}[Gupta and Newman \mbox{\cite[Lemma 2.2]{guptanewman}}]
 \label{lem: GN mod gamma5}
 Let $G$ be a $3$-Engel group and let $a,b,c\in G$.
 Then, modulo $\gamma_5(G)$, we have the following relations: 
 \begin{enumerate}
     \item 
 $[c,a,a,b][c,a,b,a][c,b,a,a][c,a,b,b][c,b,a,b][c,b,b,a] \equiv 1$;
 \item $[c,a,a,b]^2[c,a,b,a]^2[c,b,a,a]^2\equiv 1$;
 \item $[b,a,a,c]\equiv [c,b, a,a]^{-3} [c,a,a,b]^{-3}$;
 \item $[c,a,a,b]^4[c,b,a,a]^6\equiv 1$;
 \item $[c,b,a,a]^2\equiv [c,a,b,a]^4$.
 \end{enumerate}
\end{lem}

\begin{lem}[Gupta and Newman \mbox{\cite[Lemma 2.3]{guptanewman}}]
 \label{lem: GN mod gamma6}
 Let $G$ be a 3-Engel group and let $a,b,c\in G$.
 Let also $w_1=[c,a,a,b,b]$, $w_2=[c,a,b,a,b]$, $w_3=[c,a,b,b,a]$, $w_4=[c,b,a,a,b]$,
 $w_5=[c,b,a,b,a]$ and $w_6=[c,b,b,a,a]$.
 Then, modulo $\gamma_6(G)$, we have the following relations:
 \begin{enumerate}
     \item $[c,a,b,a,a]\equiv 1$;
     \item $[c,a,a,b,a]\equiv 1$;
     \item $w_3\equiv w_4\equiv w_1^{-1}w_2^{-1}$, $w_2\equiv w_5$, $w_1\equiv w_6$;
     \item $[b,a,a,b,c]\equiv 1$;
     \item $w_2\equiv w_1^3$, $w_3\equiv w_1^{-4}$;
     \item $w_1^{10}\equiv 1$.
 \end{enumerate}
\end{lem}

The following well-known lemma will be useful throughout the paper.

\begin{lem}
\label{lem: first entry}
 Let $G$ be a group, and let $x_1,\ldots,x_n,y\in G$.
 Then,
 $$
 [x_1,\ldots,x_n,y]
 \in
 \langle [y,x_{\sigma(1)},\ldots,x_{\sigma(n)}]\mid \sigma\in S_n\rangle\gamma_{n+2}(G),
 $$
 where  $S_n$ is the symmetric group of degree $n$.
 \end{lem}

\section{Powerful \texorpdfstring{$3$}{3}-Engel metabelian groups}
\label{sec: metabelian}

This section is devoted to study powerful $3$-Engel metabelian groups.
We start by proving our first main result for $n$-Engel groups with cyclic derived subgroups.

\begin{proof}[Proof of Theorem \ref{thm: derived cyclic}]
Let $G=\langle x_1,\ldots,x_r\rangle$ and $G'=\langle y\rangle$, where $r\ge 2$.
 For $1\le i<j\le r$ and $1\le k\le r$, write
 \begin{equation}
    \label{eq: relations cyclic}
    [x_i,x_j]=y^{p^{\gamma(i,j)}\delta(i,j)}\ \ \text{ and }\ \ [y,x_k]=y^{p^{\alpha(k)}\beta(k)},
 \end{equation}
 where $\gamma(i,j),\alpha(k)\ge 0$ and $p$ does not divide $\delta(i,j)$ and $\beta(k)$.
 We assume $\gamma_{n+2}(G)=1$, and claim that
 $$
 [x_{i_1},\ldots,x_{i_{n+1}}]=1
 $$
 for any $1\le i_1,\ldots,i_{n+1}\le r$.
 The relations in (\ref{eq: relations cyclic}) give
 \begin{equation}
    \label{eq: explicit cyclic}
    [x_{i_1},\dots, x_{i_{n+1}}]=y^{p^{\gamma(i_1,i_2)+\alpha(i_3)+\dots+\alpha(i_{n+1})}\delta(i_1,i_2)\beta(i_3)\cdots \beta(i_{n+1})}.
 \end{equation}
 Let $1\le j\le n$ be such that $\alpha(i_j)=\min\{\alpha(i_1), \alpha(i_2),\dots, \alpha(i_{n+1})\}$.
 We will distinguish two cases in turn: $j \in \{1,2\}$ and $j\geq 3$.

 Suppose first $j=1$. Since the group $G$ is $n$-Engel, it follows that
 $$
 1=[x_{i_2},x_{i_1},\overset{n}{\ldots}, x_{i_1}]=y^{-p^{\gamma(i_1,i_2)+(n-1)\alpha(i_1)}\delta(i_1,i_2)\beta(i_1)^{n-1}},
 $$
 so that $y^{p^{\gamma(i_1,i_2)+(n-1)\alpha(i_1)}}=1$.
 Since $(n-1)\alpha(i_1)\le \alpha(i_3)+\cdots+\alpha(i_{n+1})$, we deduce that $[x_{i_1},\dots, x_{i_{n+1}}]=1$, as claimed. If $j=2$, then $[x_{i_1},\dots, x_{i_{n+1}}]=[x_{i_2},x_{i_1},\dots, x_{i_{n+1}}]^{-1}$, which follows from the case $j=1$.

 Now, suppose $j\ge 3$. It immediately follows from (\ref{eq: explicit cyclic}) that
 \begin{equation}
    \label{eq: order}
    [x_{i_1},\dots, x_{i_{n+1}}]
    =
    [x_{i_1},x_{i_2},x_{i_{\sigma(1)+2}},\dots, x_{i_{\sigma(n-1)+2}}]
 \end{equation}
 for any $\sigma\in S_{n-1}$, so we may further assume that $j=3$.
 Now, since $\gamma_{n+2}(G)=1$, the Hall-Witt identity gives
 $$
 [x_{i_1},x_{i_2},x_{i_3},\ldots,x_{i_{n+1}}]=[x_{i_3},x_{i_2},x_{i_1},\ldots,x_{i_{n+1}}][x_{i_3},x_{i_1},x_{i_2},\ldots,x_{i_{n+1}}]^{-1},
 $$
 and the result follows from the previous case.
 
 This bound cannot be further sharpened due to Example \ref{ex: cyclic} below.
\end{proof}

\begin{example}
\label{ex: cyclic}
Let $p$ be either an odd prime or $4$, and define $G=\langle a,b\mid [a,b]=a^p,a^{p^{n}}=b^{p^{n-1}}=1\rangle$.
Then $G$ is a powerful $n$-Engel group with cyclic derived subgroup of nilpotency class $n$.
\end{example}

If the derived subgroup of our $3$-Engel group $G$ is not cyclic, then the search of best upper bounds for the nilpotency class of $G$ is much more involved.
However, the following result due to Gupta and Newman gives some insights for the structure of $G$ when $G'$ is abelian.

\begin{thm}[Gupta and Newman \mbox{\cite{GuNe66}}]
 \label{thm: gamma4^16=1}
 Let $G$ be a metabelian $3$-Engel group.
 Then we have
 $$
 \gamma_4(G)^{16}=\gamma_5(G)^2=1.
 $$
\end{thm}

As a direct consequence of this theorem, if $G$ is a $3$-Engel metabelian $p$-group, then it is nilpotent of class at most $3$ unless $p=2$.
This bound is clearly best possible, as all groups of nilpotency class $3$ are $3$-Engel.
The study of the nilpotency class of metabelian $3$-Engel powerful groups will be then reduced to the case $p=2$.
In such a case, we have
$$
\gamma_6(G)=[\gamma_5(G), G] \leq \gamma_5(G)^{4}=1,
$$
where the last inequality holds by Shalev’s Interchange Lemma \cite{Shalev1993OnAF}. Therefore, the nilpotency class of $G$ can be at most 5.
As we show in Theorem \ref{thm: metabelian}, this bound can be further improved.
Before proving it, we need the following lemma.

\begin{lem}
 \label{lem:metabsquarecomm}
 Let $G$ be a metabelian $3$-Engel $2$-group.
 Then for all $a,b,c,d \in G$ we have
 $$
 [a,b,c,d]^2[a,c,b,d]^2[a,d,c,b]^2\in\gamma_5(G).
 $$
\end{lem}
\begin{proof}
 We may assume $\gamma_5(G)=1$.
 Then, Lemma \ref{lem: GN mod gamma5}(i) yields
 $$
 1=[a,bc,bc,d]^2[a,d,bc,bc][a,bc,d,d][a,d,d,bc]^2.
 $$
 Expanding the commutators and applying Lemma \ref{lem: GN mod gamma5}(i) again, we obtain
 \begin{equation*}
 \begin{split}
     1=&[a,b,b,d]^2[a,d,b,b][a,b,d,d][a,d,d,b]^2
     [a,c,c,d]^2[a,d,c,c][a,c,d,d][a,d,d,c]^2\\
     &[a,b,c,d]^2[a,c,b,d]^2[a,d,b,c][a,d,c,b]
     =[a,b,c,d]^2[a,c,b,d]^2[a,d,c,b]^2,
 \end{split}
 \end{equation*}
 as desired.
\end{proof}

\begin{proof}[Proof of Theorem \ref{thm: metabelian}]
 By Theorem \ref{thm: gamma4^16=1}, we only need to show the result for $p=2$.
 Therefore, since $G$ is powerful, we can find generators $x_1,\ldots,x_n$ of $G$ such that $G'=\langle x_1^{4\alpha_1}, \dots, x_k^{4\alpha_k} \rangle$, with $1 \leq k \leq n$, and $\alpha_1,\dots,\alpha_k\in\Z$. Also, we can assume $k>1$ by Theorem \ref{thm: derived cyclic}.
 Thus, it suffices to show that $[x_i^{4\alpha_i}, a, b, c]=1$ for all $a,b,c, \in G$.
 For that purpose, we may assume $\gamma_5(G)^2=\gamma_6(G)=1$, so by the Hall-Petresco identity, we obtain
 $$
 [x_i^{4\alpha_i},a,b,c]=[x_i,a,b,c]^{4\alpha_i}
 $$
 for all $1\le i\le k$.
 Now, from Lemma \ref{lem:metabsquarecomm}, we deduce
 \begin{align*}
  [x_i,a,b,c]^{4\alpha_i}&=([a,x_i,b,c]^2)^{-2\alpha_i}\\
  &=([a,b,x_i, c]^{-2}[a,c,b,x_i]^{-2}g)^{-2\alpha_i}\\
  &=[a,b,x_i, c]^{4\alpha_i}[a,c,b,x_i]^{4\alpha_i}
 \end{align*}
 for some $g\in\gamma_5(G)$.
 Since, again by the Hall-Petresco identity,
 $$
 [a,b,x_i, c]^{4\alpha_i}[a,c,b,x_i]^{4\alpha_i}=[a,b,x_i^{4\alpha_i}, c][a,c,b,x_i^{4\alpha_i}]=1,
 $$
 the theorem follows.
\end{proof}

As pointed out in Remark \ref{rem: p odd or r=3}, we now show that the upper bound obtained for the nilpotency class of powerful $3$-Engel metabelian $2$-groups remains the same when the rank of $G$ is $3$, even if we relax the condition that $G$ is powerful.

\begin{lem}
 \label{lem: metabelian 3-generator}
 Let $G = \langle a, b, c \rangle$ be a metabelian 3-Engel 2-group.
 Then $\gamma_5(G)=1$.
\end{lem}
\begin{proof}
 We are going to show that all commutators of length 5 in $\{a,b,c\}$ are trivial. Since $G$ is $3$-Engel, Lemma \ref{lem: GN mod gamma6} already shows that any commutator with a triple entry is trivial.
 Therefore, by Lemma \ref{lem: first entry}, and since $G$ is metabelian, it remains to show that commutators of the form $[x,y,y,z,z]$ with $x,y,z\in\{a, b,c\}$ are trivial.
 Since $[x,y,y,z,z]=[x,y,z,z,y]$, and, additionally, by Lemma \ref{lem: GN mod gamma6}, we have $[x,y,z,z,y]=[x,y,y,z,z]^{-4}$, it follows that $[x,y,y,z,z]^5=1$.
 Since $G$ is a $2$-group, the result follows.
\end{proof}

We end this section by noting that this bound is best possible, as the following example shows.

\begin{example}
 \label{ex: p=2,r=3,powerful}
 Let $p=2$ and let $G=\langle a,b,c\rangle$ be a group with the following presentation:
 \begin{center}
 \begin{tabular}{lll}
    $[a,b]=a^{-p^5}$, & $[a,c]=b^{3p^6}a^{p^4}$, & $[b,c]=b^{p^7}a^{p^2}$,\\
    $[a^{p^2},b^{p^6}]=1$, & $[a^{p^2},b]=a^{-p^7}$, & $[a^{p^2},c]=a^{-p^2}b^{3p^8}a^{17p^2}$,\\
    $[b^{p^6},a]=a^{p^{11}}$, & $[b^{p^6},c]=a^{p^8}$, & \\
    $a^{p^{12}}=1$, & $b^{p^{11}}=1$, & $c^{p^{10}}=1$. 
 \end{tabular}
 \end{center}
 Then, $G$ is a powerful $3$-Engel $2$-group of nilpotency class $4$.
\end{example}

\section{Powerful \texorpdfstring{$3$}{3}-Engel groups of rank \texorpdfstring{$3$}{3}}

\label{sec: powerful rank 3}

The aim of this section is proving Theorem \ref{thm: non-metabelian}(i), (iii) and (iv).
For that purpose, let $G$ be a \begin{equation}
\label{conditions}
\text{powerful }3\text{-Engel }p\text{-group such that }\left\{
\begin{aligned}
    \gamma_5(G)=\gamma_4(G)^p=1 & \text{ if }p\neq 2\\
    \gamma_6(G)=\gamma_5(G)^2=1 & \text{ if }p=2.
\end{aligned}
\right.
\end{equation}

\noindent We will fix for the rest of the section $a,b,c\in G$ and $i,j,k\in\mathbb{N}$ with $i\le j\le k$ such that $G=\langle a,b,c\rangle$ and $G'=\langle a^{p^i},b^{p^j},c^{p^k}\rangle$.
Also, since $G$ is powerful, let $r,s,t,u,v,w,\alpha,\beta,\gamma\in\Z$ be such that
\begin{equation}
\label{eq: notation}
 [a,b]=a^{rp^{i}}b^{sp^{j}}c^{tp^{k}}, \quad [c,a]=a^{up^{i}}b^{vp^{j}}c^{wp^{k}}, \quad [c,b]=a^{\alpha p^{i}}b^{\beta p^{j}}c^{\gamma p^{k}}.
\end{equation}
The following remark will be frequently used.

\begin{rem}\label{rem: estrella}
Suppose that $\gamma_n(G)^p=\gamma_{n+1}(G)=1$ for some $n\ge 2$.
If $x^{p^i}\in G'$ for some $x\in G$ and $i\ge 0$, then
$$
1=[x^{p^i},y_1,{\ldots},y_{n-2}]^p=[x,y_1,\ldots,y_{n-2}]^{p^{i+1}}
$$
for all $y_1,\ldots,y_{n-2}\in G$.
\end{rem}

First we will establish Theorem \ref{thm: non-metabelian}(i) and (iii), i.e., we will show that for $G$ as in (\ref{conditions}) with $p\neq 2$, we have $\gamma_4(G)=1$.
For that purpose, we first need the following two lemmas.

\begin{lem}
\label{lem: i<j}
Let $G$ be as in (\ref{conditions}) with $p\neq 2$.
If $i<j\le k$, then $\gamma_4(G)=1$.
\end{lem}
\begin{proof}
By \cite[Lemma 2.1]{guptanewman}, commutators of length $4$ in $\{a,b,c\}$ with only two different entries are all trivial.
We focus, then, on commutators of length $4$ in $\{a,b,c\}$ with $3$ different entries.

Observe that
$$
[x,y,a,z]=[a^{\sigma_1 p^i}b^{\sigma_2p^j}c^{\sigma_3p^k},a,z]=[a,b,z]^{-\sigma_2p^j}[a,c,z]^{-\sigma_3p^k}
$$
for every $x,y,z\in G$, where $\sigma_1,\sigma_2,\sigma_3\in\Z$.
Since $i<j,k$ and $n=4$, it follows from Remark \ref{rem: estrella} that $[x,y,a,z]=1$, so that all commutators of length $4$ with $a$ in the third entry are trivial.
In particular, $[c,b,a,b]$, $[c,b,a,a]$, and $[b,c,a,c]$ are trivial, and so, from (iv) and (v) of Lemma \ref{lem: GN mod gamma5}, it follows that all the possible combinations with $a,b$ and $c$ produce trivial commutators, as desired.
\end{proof}

\begin{lem}
\label{lem: i=j<k}
Let $G$ be as in (\ref{conditions}) with $p\neq 2$.
If $i=j<k$, then $\gamma_4(G)=1$.
\end{lem}
\begin{proof}
By \cite[Lemma 2.1]{guptanewman}, commutators of length $4$ in $\{a,b,c\}$ with only two different entries are all trivial.

Let us suppose first that all commutators of length $4$ in $\{a,b,c\}$ with exactly one entry in $c$ are trivial.
Then, by (iv) and (v) of Lemma \ref{lem: GN mod gamma5}, we only need to show that one of $[c,a,b,c]$, $[c,b,a,c]$ or $[a,b,c,c]$ is trivial.
Following the notation in (\ref{eq: notation}), we have
\begin{align}
    \label{eq: cabc}
    [c,a,b,c]&=[a,b,c]^{up^i}[b,c,c]^{-wp^k}=[a,b,c]^{up^i},\\
    \label{eq: cbac}
    [c,b,a,c]&=[b,a,c]^{\beta p^i}[a,c,c]^{-\gamma p^k}=[b,a,c]^{\beta p^i},\\
    \label{eq: abcc}
    [a,b,c,c]&=[a,c,c]^{rp^i}[b,c,c]^{sp^i},
\end{align}
where $[b,c,c]^{-wp^k}=[a,c,c]^{-\gamma p^k}=1$ by Remark \ref{rem: estrella} and $k>i$.
Note that, similarly, we have
\begin{align*}
    1&=[a,b,b,c]=[a,b,c]^{rp^i}[c,b,c]^{tp^k}=[a,b,c]^{rp^i},\\
    1&=[a,b,a,c]=[b,a,c]^{sp^i}[c,a,c]^{tp^k}=[b,a,c]^{sp^i}.
\end{align*}
Thus, if $p\doesnotdivide r$ or $p\doesnotdivide s$, then $[a,b,c]^{p^i}=1$ or $[b,a,c]^{p^i}=1$, and so (\ref{eq: cabc}) or (\ref{eq: cbac}) is trivial, respectively.
We hence may assume that $p$ divides both $r$ and $s$, but then (\ref{eq: abcc}) is trivial by Remark \ref{rem: estrella}, as desired.

Remove now the assumption that all commutators of length $4$ in $\{a,b,c\}$ with exactly one entry in $c$ are trivial.
By Lemma \ref{lem: first entry}, we only need to show that commutators of the form $[c,x,y,z]$, where $x,z,y\in\{a,b\}$, are trivial.
Also, we will only consider commutators of length $4$ with $2$ entries in $a$.
That commutators with $2$ entries in $b$ are also trivial will follow by a symmetric argument.
Note, by Remark \ref{rem: estrella}, that $[c,x,y]^{p^k}=1$ for every $x, y\in G$.
We will use this fact without special mention. Keeping also in mind the notation in (\ref{eq: notation}), we distinguish between the following different cases.

\vspace{0.2cm}
\noindent\underline{Case 1}: $p\doesnotdivide r$.
Then $1=[a,b,b,a]=[a,b,a]^{rp^i}$, and so $[a,b,a]^{p^i}=1$.
Now, $[c,b,a,a]=[b,a,a]^{\beta p^i}=1$, and thus $[c,a,b,a]=[c,a,a,b]=1$ by Lemma \ref{lem: GN mod gamma5}(iv) and (v).

\vspace{0.2cm}
\noindent\underline{Case 2}: $p\doesnotdivide s$.
Then $1=[a,b,a,a]=[b,a,a]^{sp^i}$, and so $[b,a,a]^{p^i}=1$.
Therefore, $[c,a,b,a]=[a,b,a]^{u p^i}=1$, and we conclude again by Lemma \ref{lem: GN mod gamma5} that $\gamma_4(G)=1$.

\vspace{0.2cm}
\noindent\underline{Case 3}: $p|r$, $p|s$.
Then, by Remark \ref{rem: estrella}, $[a,b,c,a]=[a,c,a]^{rp^i}[b,c,a]^{sp^j}\leq \gamma_4(G)^p=1$.
As a consequence, $[c,b,a,a]=[c,a,b,a]=[c,a,a,b]=1$ by Lemma \ref{lem: GN mod gamma5}(iv) and (v).
\end{proof}

The previous two lemmas immediately imply that, for $G$ as in the statements, if $G'$ can be generated by $2$ elements, then $\gamma_4(G)=1$.
Indeed, in such a case, there exist generators $a,b,c$ such that $G'=\langle a^{p^i},b^{p^j}\rangle$, and we may choose $k$ as big as we want.
The following theorem establishes the result for the general case (with $p\neq 2$) when $G'$ is $3$-generator, settling \ref{thm: non-metabelian}(i) and (iii).


\begin{thm}
\label{thm: r=3, p neq2}
Let $G$ be as in (\ref{conditions}) with $p\neq 2$.
Then $\gamma_4(G)=1$.
\end{thm}
\begin{proof}
By Lemmas \ref{lem: i<j} and \ref{lem: i=j<k} we may assume that $i=j=k$.
We will show that all commutators of length $4$ in $\{a,b,c\}$ are trivial.
As before, commutators with only two different entries are all trivial, so we only consider commutators with $3$ different entries.
We claim that commutators of length $4$ with exactly two entries of $a$ are trivial.
Then, a completely symmetric argument can show that commutators with two entries in $b$ or $c$ are also trivial, so we will be done.
Note that by Lemma \ref{lem: first entry} it suffices to consider commutators where $c$ is in the first entry.
As a matter of fact, since $p\neq 2$, we will just show that one of $[c,b,a,a]$ or $[c,a,b,a]$ is trivial, and then conclude from Lemma \ref{lem: GN mod gamma5}(iv) and (v) that $\gamma_4(G)=1$.
We follow the notation in (\ref{eq: notation}) and distinguish between several cases depending on whether $p$ divides $u, v$ or $w$.


\vspace{0.2cm}
\noindent\underline{Case 1}: $p|w$.
Then we have $1=[c,a,a,a]=[b,a,a]^{vp^i}$, $[c,a,b,a]=[a,b,a]^{up^i}$ and $[c,a,a,b]=[b,a,b]^{vp^i}$ by Remark \ref{rem: estrella}.
We consider further subcases.

\vspace{0.2cm}
\underline{Case 1.1}: $p|u$.
Then we have $[c,a,b,a]=1$.

\vspace{0.2cm}
\underline{Case 1.2}: $p\doesnotdivide v$.
Then $[b,a,a]^{p^i}=1$, so $[c,a,b,a]=1$.

\vspace{0.2cm}
\underline{Case 1.3}: $p\doesnotdivide u$, $p|v$.
In this case $[c,a,a,b]=1$, so if $p\neq 3$ we finish the argument as before by Lemma \ref{lem: GN mod gamma5}(iv).
Assume then $p=3$.
Note that $1=[c,a,c,a]=[a,c,a]^{u p^i}$, and since $p\doesnotdivide u$ we obtain $[a,c,a]^{p^i}=1$.
If $p|\beta$, then $[c,b,a,a]=[c,a,a]^{\gamma p^i}=1$ and we are done.
If $p\doesnotdivide \beta$, then observe from Lemma \ref{lem: GN mod gamma5}(iv) that $1=[c,b,c,a]=[b,c,a]^{\beta p^i}$, so $[a,b,c,a]=[b,c,a]^{sp^i}=1$.
Lemma \ref{lem: GN mod gamma5}(v) gives now $[c,b,a,a]=1$, as desired.

\vspace{0.3cm}
\noindent\underline{Case 2}: $p\doesnotdivide w$.

\vspace{0.2cm}
\underline{Case 2.1}: $p|u$, $p\doesnotdivide v$.
Then $1=[c,a,c,a]=[b,c,a]^{vp^i}$, so $[b,c,a]^{p^i}=1$.
Hence, $[c,a,b,a]=[a,b,a]^{up^i}[c,b,a]^{wp^i}=1$.

\vspace{0.2cm}
\underline{Case 2.2}: $p\doesnotdivide u$, $p|v$.
Observe that $1=[c,a,c,a]=[a,c,a]^{up^i}$ which implies that $[a,c,a]^{p^i}=1$.
Now, if $p\doesnotdivide s$, then $1=[a,b,a,a]=[b,a,a]^{sp^i}$ implies $[b,a,a]^{p^i}=1$.
Thus, $[c,b,a,a]=[b,a,a]^{\beta p^i}[c,a,a]^{\gamma p^i}=1$, and we are done.
If $p|s$, then $[a,b,c,a]=[a,c,a]^{rp^i}[b,c,a]^{sp^i}=1$, and we conclude from Lemma \ref{lem: GN mod gamma5}(v) that $\gamma_4(G)=1$.

\vspace{0.2cm}
\underline{Case 2.3}: $p|u$, $p|v$, $p\neq 3$.
Then $1=[c,a,a,a]=[c,a,a]^{wp^i}$, so $[c,a,a]^{p^i}=1$.
Therefore $1=[a,b,a,a]=[b,a,a]^{sp^i}[c,a,a]^{tp^i}=[b,a,a]^{sp^i}$.

Now, if $p\doesnotdivide s$, then $[b,a,a]^{p^i}=1$, so $[c,b,a,a]=[b,a,a]^{\beta p^i}[c,a,a]^{\gamma p^i}=1$ and we are done.

If $p|s$ but $p\doesnotdivide t$, then $1=[a,b,a,b]=[b,a,b]^{sp^i}[c,a,b]^{tp^i}=[c,a,b]^{tp^i}$, so $[c,a,b]^{p^i}=1$.
Thus, $[c,a,a,b]=[c,a,b]^{wp^i}=1$, and since $p\neq 3$, we finish by Lemma \ref{lem: GN mod gamma5}(iv). 

If $p|s$, $p|t$ but $p\doesnotdivide r$, then $1=[a,b,b,a]=[a,b,a]^{rp^i}$, so $[a,b,a]^{p^i}=1$, and therefore, $[c,b,a,a]=[b,a,a]^{\beta p^i}[c,a,a]^{\gamma p^i}=1$.

Finally, if $p|s$, $p|t$ and $p|r$, then it follows that $[a,b]\in (G')^p$, and so $G'$ can be generated by two elements.
The result follows then from the discussion after Lemma \ref{lem: i=j<k}.

\vspace{0.2cm}
\underline{Case 2.4}: $p|u$, $p|v$, $p=3$.
As in the previous case, we have $[c,a,a]^{p^i}=1$.
Now, since $p=3$, Lemma \ref{lem: GN mod gamma5}(iv) yields $1=[c,b,c,a]=[a,c,a]^{\alpha p^i}[b,c,a]^{\beta p^i}=[b,c,a]^{\beta p^i}$.

If $p\doesnotdivide \beta$, then $[b,c,a]^{p^i}=1$, and so $[c,a,b,a]=[a,b,a]^{up^i}[c,b,a]^{wp^i}=1$.
If $p|\beta$, then $[c,b,a,a]=[b,a,a]^{\beta p^i}[c,a,a]^{\gamma p^i}=1$.

\vspace{0.2cm}
\underline{Case 2.5}: $p\doesnotdivide u$, $p\doesnotdivide v$.
Then $1=[c,a,a,a]=[b,a,a]^{vp^i}[c,a,a]^{wp^i}$, so that $[a,b,a]^{vp^i}=[a,c,a]^{-wp^i}$.
Similarly, $1=[c,a,c,a]=[a,c,a]^{up^i}[b,c,a]^{vp^i}$, so that $[c,b,a]^{vp^i}=[a,c,a]^{up^i}$.
Now,
$$
[c,a,b,a]^v=[a,b,a]^{vup^i}[c,b,a]^{vwp^i}=[a,c,a]^{-uwp^i}[a,c,a]^{uwp^i}=1,
$$
and we finish the proof.
\end{proof}

The bound for the nilpotency class of $G$ in Theorem \ref{thm: r=3, p neq2} is clearly best possible, since every nilpotent group of class $3$ is $3$-Engel.
We conclude this section by proving Theorem \ref{thm: non-metabelian}(iv).

\begin{thm}
\label{thm: r=3, p=2}
Let $G$ be as in (\ref{conditions}) with $p=2$.
Then $\gamma_5(G)=1$.
\end{thm}
\begin{proof}
From Lemma \ref{lem: GN mod gamma6}, every commutator of length $5$ such that $3$ of its entries are equal is trivial.
Also, from Lemma~\ref{lem: first entry}, we may only consider commutators whose non repeated entry is in the first position.
We claim that all commutators of length $5$ with $c$ in the first entry and $2$ entries of $a$ and $b$ are trivial.
Then, a completely symmetric argument will complete the proof.

Following the notation in (\ref{eq: notation}), we have $[c,a,a,b,b]=[c,a,b,b]^{w 2^k}$, so if $p| w$, then $[c,a,a,b,b]=1$ by Remark \ref{rem: estrella}, and the claim follows from Lemma \ref{lem: GN mod gamma6}(iii) and (v).

If $p\doesnotdivide  w$, note that $1=[c,a,a,b,a]=[b,a,b,a]^{v 2^j}[c,a,b,a]^{w 2^k}$ and that $[b,a,b,a]^2=1$ by \cite[Theorem 2.1]{guptanewman}.
Hence $[c,a,b,a]^{2^k}=1$, so that $[c,b,a,b,a]=[c,a,b,a]^{\gamma 2^{k}}=1$,
and the result follows again by Lemma \ref{lem: GN mod gamma6}.
\end{proof}

Note that the upper bound in Theorem \ref{thm: r=3, p=2} is best possible by Example \ref{ex: p=2,r=3,powerful}.

\section{Powerful \texorpdfstring{$3$}{3}-Engel groups of rank greater than \texorpdfstring{$3$}{3} and the proof of Theorem \ref{thm: non-metabelian}}

\label{sec: powerful rank r}

The following theorem due to Gupta and Newman, together with Examples \ref{ex: p=3} and \ref{ex: p>3} after it, establish Theorem \ref{thm: non-metabelian}(ii).

\begin{thm}[Gupta and Newman \mbox{\cite[Theorem 4.4]{guptanewman}}]
\label{thm: gamma5^20=1}
Let $G$ be a $3$-Engel group.
Then, $\gamma_5(G)^{20}=1$.
\end{thm}

\begin{example}
\label{ex: p=3}
Let $p=3$ and let $G=\langle a,b,c,d\rangle$ be a group with the following presentation:

\ 

\begin{center}
    \begin{tabular}{lll}
         $[b,c]=b^{2p^2}$, & $[b,d]=a^{p}$, & $[c,d]=d^{p^2}$,\\
         $[a,b]=1$, & $[a,c]=a^{p^2}c^{-p^4}$, & $[a,d]=c^{p^2}$,
    \end{tabular}
\end{center}
\begin{center}
    \begin{tabular}{llll}
         $a^{p^5}=1$, & $b^{p^6}=1$ & $c^{p^6}=1$, & $d^{p^4}=1$. 
    \end{tabular}
\end{center}

\ 

\noindent Then $G$ is a powerful $3$-Engel $3$-group of nilpotency class $4$.
\end{example}

\begin{example}
\label{ex: p>3}
Let $p>3$ and let $G=\langle a,b,c,d\rangle$ be a group with the following presentation:

\ 

\begin{center}
    \begin{tabular}{lll}
        $[b,c]=c^{p^4}$, & $[b,d]=b^{p^3}$, & $[c,d]=a^p$,\\
        $[a,b]=a^{-p^4}c^{p^6}$, & $[a,c]=1$, & $[a,d]=a^{\alpha p^3}c^{\gamma p^5}$,,
    \end{tabular}
\end{center}
\begin{center}
    \begin{tabular}{lllll}
        $c^{p^9}=a^{-3p^7}$, & $a^{p^9}=1$, & $b^{p^9}=1$, & $c^{p^{11}}=1$, & $d^{p^8}=1$,
    \end{tabular}
\end{center}

\noindent with $\alpha=-5r$, $\gamma=-11s$, where $r$ and $s$ are taken such that $3r\equiv 1\pmod{p^9}$ and $9s\equiv 1\pmod{p^{11}}$.
Then $G$ is a powerful $3$-Engel $p$-group of nilpotency class $4$.
\end{example}

Note that Theorem \ref{thm: gamma5^20=1} shows that $\gamma_5(G)=1$ if $p\not\in\{2,5\}$.
The following result shows that the same is true even if the prime is $2$ or $5$, as long as the rank of the group is small.
This proves Theorem \ref{thm: non-metabelian}(v).

\begin{thm}
\label{thm: r=4,5}
Let $G$ be a powerful $3$-Engel $p$-group of rank at most $5$.
Then $\gamma_5(G)=1$.
\end{thm}
\begin{proof}
The result follows immediately from Theorem \ref{thm: gamma5^20=1} if $p\neq 2,5$, so let $p\in\{2,5\}$.
We assume $\gamma_6(G)=\gamma_5(G)^p=1$, and we first claim that $[c,a,a,b,b]=1$ for every $a,b,c\in G$.
Fix thus $a,b,c\in G$ and consider the subgroup
$$
H=\langle [c,a,a],[c,b,b],[c,a,b],[c,b,a],[a,b,b],[b,a,a]\rangle.
$$
Since $G$ is powerful of rank $5$, the rank of $H$ is at most $5$, and since $\gamma_3(G)^p\le Z_2(G)$, there exist $\alpha,\beta,\gamma,\delta,\epsilon,\zeta\in\mathbb{F}_p$ not all equal to zero such that
$$
[c,a,a]^{\alpha}[c,b,b]^{\beta}[c,a,b]^{\gamma}[c,b,a]^{\delta}[b,a,a]^{\epsilon}[a,b,b]^{\zeta}\equiv 1\pmod{Z_2(G)}.
$$
If $\alpha\neq 0$, then $[c,a,a]$ can be written as a product of the rest of commutators modulo $Z_2(G)$, and since commutators of length $5$ with three repeated entries are trivial by Lemma \ref{lem: GN mod gamma6}(i) and (ii), it follows that $[c,a,a,b,b]=1$.
We may assume then that $\alpha=0$, and by Lemma \ref{lem: GN mod gamma6}(iii), we can repeat the same argument with $[c,b,b]$ and assume $\beta=0$.

Now, if $p=2$, then Lemma \ref{lem: GN mod gamma6}(v) yields $[c,b,a,a,b]=1$, and so, if $\gamma\neq 0$, it then follows that $[c,a,b,a,b]=1$.
Thus, Lemma \ref{lem: GN mod gamma6}(v) gives $[c,a,a,b,b]=1$, so we may assume $\gamma=0$.
A symmetric argument shows that we may also assume that $\delta=0$.
Finally, if $\epsilon\neq 0$, then Lemma \ref{lem: GN mod gamma5}(iii) yields
$$
[c,a,a,b,b]=[b,c,a,a,b][b,a,a,c,b]=[c,b,a,a,b]=1,
$$
and similarly if $\zeta\neq0$.

If $p=5$, then $[c,a,b,b,a]=[c,a,a,b,b]$ and $[c,a,b,a,b]=[c,a,a,b,b]^3$ by Lemma \ref{lem: GN mod gamma6}(iii) and (v).
Then, if $\gamma\neq 0$, we must have, on the one hand,
$$
[c,a,a,b,b]^3=[c,a,b,a,b]=[c,b,a,a,b]^{-\delta\gamma^{-1}}=[c,a,a,b,b]^{-\delta\gamma^{-1}},
$$
so that $-\delta\gamma^{-1}=3$.
On the other hand,
$$
[c,b,a,b,a]=[c,a,b,a,b]=[c,b,a,a,b]^{-\delta\gamma^{-1}}=[c,a,b,b,a]^{-\delta\gamma^{-1}}=[c,b,a,b,a]^{(-\delta\gamma^{-1})^2},
$$
so that $(-\delta\gamma^{-1})^2=1$, a contradiction.
Hence $\gamma=0$, and similarly we have $\delta=0$ as well.
Finally, if $\epsilon\neq 0$, then we have by Lemma \ref{lem: GN mod gamma5}(iii)
$$
[c,a,a,b,b]=[b,c,a,a,b]^{-3}[b,a,a,c,b]^{-3}=[c,b,a,a,b]^3=[c,a,a,b,b]^3.
$$
Therefore $[c,a,a,b,b]=1$, and similarly if $\zeta\neq 0$, so the claim follows.

Now, for every $a,b,c,d\in G$ we have
$$
1=[a,bc,bc,d,d]=[a,b,c,d,d][a,c,b,d,d],
$$
so that $[a,b,c,d,d]=[c,a,b,d,d]$.
The Hall-Witt identity thus gives
$$
1=[a,b,c,d,d][c,a,b,d,d][b,c,a,d,d]=[a,b,c,d,d]^3,
$$
and then $[a,b,c,d,d]=1$.

If $p=5$, then we apply Lemma \ref{lem: GN mod gamma5}(iv) twice, so that we obtain $[a,d,d,b,c]=1$ for every $a,b,c,d\in G$.
Applying again the same argument, we conclude that $[a,b,c,d,e]=1$ for every $a,b,c,d,e\in G$, as desired.

For the case $p=2$, observe that for every $a,b,c,d,e\in G$ we have
$$
1=[a,b,c,de,de]=[a,b,c,d,e][a,b,c,e,d],
$$
so that $[a,b,c,d,e]=[a,b,c,e,d]$.
The Hall-Witt now identity gives
$$
1=[a,b,c,d,e][d,e,[a,b,c]][e,[a,b,c],d]=[d,e,[a,b,c]],
$$
and hence $[G',\gamma_3(G)]=1$.
Thus, by the three subgroup lemma, it follows that 
$$
[G',G',G]\le [G',G,G'][G,G',G']=[\gamma_3(G),G']=1,
$$
and therefore, $[a,b,c,d,e]=[a,b,d,c,e]$ for every $a,b,c,d,e\in G$.
In particular, this shows that
$$
[a,b,c,c,d]=[a,b,c,d,c]=[a,b,d,c,c]=1
$$
for every $a,b,c,d\in G$.

We now show that $[a,b,b,c,d]=1$ for every $a,b,c,d\in G$.
Fix $a,b,c,d\in G$.
By Lemma \ref{lem: GN mod gamma5}(iii), since $[c,a,b,b,d]=1$, we obtain $[a,b,b,c,d]=[c,b,b,a,d]$. Thus, we deduce that
\begin{equation}
\label{eq: equalities} 
[a,b,b,c,d]=[a,b,c,b,d]=[a,b,d,b,c]=[c,b,b,a,d]=[c,b,d,b,a]=[d,b,b,a,c].
\end{equation}
Let
$$
H=\langle[a,b,b],[a,b,c],[a,b,d],[c,b,b],[c,b,d],[d,b,b]\rangle.
$$
If $H\leq Z_2(G)$ then $[a,b,b,c,d]=1$, so assume $H\not\leq Z_2(G)$. Note that $\gamma_3(G)^2\leq Z_2(G)$, so since $G$ is powerful of rank $5$, there must exist $\alpha,\beta,\gamma,\delta,\epsilon,\zeta\in\{0,1\}$, not all $0$, such that
$$
[a,b,b]^\alpha[a,b,c]^\beta[a,b,d]^\gamma[c,b,b]^\delta[c,b,d]^\epsilon[d,b,b]^\zeta \equiv 1 \pmod{Z_2(G)}.
$$
Now, if $\alpha\neq 0$, then we have
$$
[a,b,b]\equiv[a,b,c]^\beta[a,b,d]^\gamma[c,b,b]^\delta[c,b,d]^\epsilon[d,b,b]^\zeta \pmod{Z_2(G)},
$$
and substituting it in $[a,b,b,c,d]$, we obtain $[a,b,b,c,d]=1$.
If any of the $\beta,\gamma,\delta,\epsilon,\zeta$ is non-zero, then, in view of (\ref{eq: equalities}), the same argument applies.
Therefore, $[a,b,b,c,d]=1$ for any $a,b,c,d\in G$, as we wanted.

Now, $1=[a,bc,bc,d,e]=[a,b,c,d,e][a,c,b,d,e]$, or in other words, $[a,b,c,d,e]=[a,c,b,d,e]$ for any $a,b,c,d,e\in G$.
The Hall-Witt identity hence gives
$$
1=[a,b,c,d,e][b,c,a,d,e][c,a,b,d,e]=[a,b,c,d,e]^3=[a,b,c,d,e],
$$
and the result follows.
\end{proof}

Finally, in Theorem \ref{thm: gamma6=1}  and Theorem \ref{thm: r>5} below, we show that, in general, $\gamma_6(G)=1$ and that $\gamma_5(G)$ may be non-trivial if the rank of $G$ exceeds $5$.
This will prove Theorem \ref{thm: non-metabelian}(vi) and we will thus conclude the proof of our main result.

\begin{thm}\label{thm: gamma6=1} 
Let $G$ be a powerful $3$-Engel $p$-group.
Then, $\gamma_6(G)=1$.
\end{thm}
\begin{proof}
Since $\gamma_5(G)^{20}=1$ by Theorem \ref{thm: gamma5^20=1}, it follows that $\gamma_5(G)=1$ if $p\not\in\{2,5\}$ and, since $G$ is powerful, that
$$
\gamma_6(G)=[\gamma_5(G),G]\le\gamma_5(G)^{2p}=1
$$
if $p=2$ or $5$.
\end{proof}

\begin{lem}
\label{lem: nilpotent so powerful}
Let $G$ be a finite $p$-group of nilpotency class $c\le p^i$, where $i\ge 1$.
If $p$ is odd, then $G^{p^i}$ is powerful, and if $p=2$, then $G^{2^{i+1}}$ is powerful.
\end{lem}
\begin{proof}
Let $p$ be odd.
By the Hall-Petresco identity we have
$$
[x^{p^i},y^{p^i}]=[x^{p^i},y]^{p^i}c_2^{\binom{p^i}{2}}\cdots c_{p^i-1}^{\binom{p^i}{p^i-1}}c_{p^i},
$$
where $c_j\in\gamma_j(\langle [x^{p^i},y],y\rangle)\le G^{p^i}\cap\gamma_{j+1}(G)$ for $j=2,\ldots,p$.
In particular, $c_{p^i}\in\gamma_{p^i+1}(G)=1$, so that $[G^{p^i},G^{p^i}]\le (G^{p^i})^p$.
For $p=2$ the proof follows similarly since $4|\binom{2^{i+1}}{j}$ for every $2\le j\le 2^i-1$.
\end{proof}

%
\begin{thm}
\label{thm: r>5}
For $p\in \{2,5\}$, there exists a powerful $3$-Engel $p$-group
 where $\gamma_5(G)\neq 1$.
\end{thm}
\begin{proof} 
To aid the reader we will first construct an analogous Lie ring example and then guided by the Mal'cev correspondence, between Lie algebras over ${\Q}$ and ${\Q}$-powered groups, we will produce the group we want in an analogous way. \newline\newline
Let $F=\langle a,b,c\rangle_{\Q}$ be the free $\Q$-Lie algebra on $3$-generators.
Let $I$ be the ideal of $F$ $\Q$-generated by:
\begin{itemize}
    \item All commutators in $\{a,b,c\}$ of length $6$.
    \item All commutators in $\{a,b,c\}$  with $2$ entries in $a$.
    \item All commutators in $\{a,b,c\}$ with $3$ entries in $b$.
    \item All commutators in $\{a,b,c\}$ with $3$ entries in $c$.
    \item All commutators in $\{b,c\}$ with $2$ entries in $b$ and $2$ entries in $c$.
\end{itemize}
We first consider the case when $p=5$.  
With some abuse of notation, write $L=F/I=\langle a, b, c\rangle_{\Q}$, and note that $\gamma_6(L)=1$. Notice that the ${\Q}$-algebra $L$ has basis $a$, $b$, $c$, $[a,b]$, $[a,c]$, $[b,c]$, $[a,b,b]$, $[a,b,c]$, $[a,c,b]$, $[a,c,c]$, $[b,c,c]$, $[c,b,b]$, $[a,b,b,c]$, $[a,b,c,b]$, $[a,c,b,b]$, $[a,c,c,b]$, $[a,c,b,c]$, $[a,b,c,c]$, $[a,b,b,c,c]$, $[a,b,c,b,c]$, $[a,b,c,c,b]$, $[a,c,c,b,b]$, $[a,c,b,c,b]$, $[a,c,b,b,c]$. Therefore 
$\dim_{\Q}(L)=24$ and $\dim_{\Q}(\gamma_5(L))=6$.
Now consider the Lie subring of $L$ generated by $\frac{1}{p}a, \frac{1}{p}b$ and $\frac{1}{p}c$, $M=\langle \frac{1}{p}a,\frac{1}{p}b,\frac{1}{p}c\rangle_{\Z}$, and define $K=pM\subseteq L$.
Then
$$
[K,K]=[pM,pM]=p^2[M,M]\le p^2M=pK,
$$
so that $R=K/p^{5}K$ is a powerful $p$-Lie ring. 
Moreover, $\langle a,b,c\rangle_{\Z}\le K$, so the nilpotency class of $K$ is $5$.
Also in this case, we have $\dim_{\Z}(K)=24$ and $\dim_{\Z}(\gamma_5(K))=6$. The ${\Z}$-basis consists of the following elements:
$a$, $b$, $c$, $d_{1}=\frac{1}{p}[a,b]$, $d_{2}=\frac{1}{p}[a,c]$, $d_{3}=\frac{1}{p}[b,c]$, $e_{1}=\frac{1}{p^{2}}[a,b,b]$, $e_{2}=\frac{1}{p^{2}}[a,b,c]$, $e_{3}=\frac{1}{p^{2}}[a,c,b]$, $e_{4}=\frac{1}{p^{2}}[a,c,c]$, $e_{5}=\frac{1}{p^{2}}[b,c,c]$, $e_{6}=\frac{1}{p^{2}}[c,b,b]$, $f_{1}=\frac{1}{p^{3}}[a,b,b,c]$, $f_{2}=\frac{1}{p^{3}}[a,b,c,b]$, $f_{3}=\frac{1}{p^{3}}[a,c,b,b]$, $f_{4}=\frac{1}{p^{3}}[a,c,c,b]$, $f_{5}=\frac{1}{p^{3}}[a,c,b,c]$, $f_{6}=\frac{1}{p^{3}}[a,b,c,c]$, $g_{1}=\frac{1}{p^{4}}[a,b,b,c,c]$, $g_{2}=\frac{1}{p^{4}}[a,b,c,b,c]$, $g_{3}=\frac{1}{p^{4}}[a,b,c,c,b]$, $g_{4}=\frac{1}{p^{4}}[a,c,c,b,b]$, $g_{5}=\frac{1}{p^{4}}[a,c,b,c,b]$, $g_{6}=\frac{1}{p^{4}}[a,c,b,b,c]$. 

Notice that $\gamma_{5}(K)\cap p^{5}K=p\gamma_{5}(K)$ and thus $R=K/p^{5}K$ is nilpotent of class $5$ and also has dimension $24$ as an abelian $p$-group.
Now consider the verbal ideal $J$ of $K$ generated by the word $[x,y,y,y]$ and $p^5x$.
As $p>3$,
%
%
%
%
%
 $J$ is also generated by $p^{5}x$ and the multilinear word
                  $$\sum_{\sigma\in S_{3}}[x,y_{\sigma(1)},y_{\sigma(2)},y_{\sigma(3)}].$$
%
It follows that $J$ is a multigraded ideal in $a,b,c,d_{1},d_{2},d_{3},e_{1},\ldots ,e_{6},f_{1},\ldots ,f_{6},g_{1},\ldots ,g_{6}$.

Note that $e_{1},\ldots ,e_{6},f_{1},\ldots ,f_{6}, g_{1},\ldots ,g_{6}$ are in the third center and that the only 
multi-homogeneous components that are non-trivial (and not involving only $a,b,c$) are those in $\{b,c,c,d_{1}\}$, $\{b,b,c,d_{2}\}$, $\{a,b,c,d_{3}\}$. 
We consider first the two of these. 

%
 It is easy to see and well known \cite{guptanewman}, that the relators for these multi-homogeneous components are generated by 
\begin{eqnarray}
   \frac{1}{p}[a,b,b,c,c]= [d_{1},b,c,c] & = & 2[d_{1},c,b,c]=\frac{2}{p}[a,b,c,b,c]; \label{eq: relation 1} \\
\frac{1}{p}[a,b,c,c,b]=[d_{1},c,c,b] & = & -3[d_{1},c,b,c]=-\frac{3}{p}[a,b,c,b,c];  \\
   \mbox{}\frac{1}{p}[a,c,c,b,b]=[d_{2},c,b,b] & = & 2[d_{2},b,c,b]=\frac{2}{p}[a,c,b,c,b];\\
            \mbox{}\frac{1}{p}[a,c,b,b,c]=[d_{2},b,b,c] & = & -3[d_{2},b,c,b]=-\frac{3}{p}[a,c,b,c,b].\label{eq: relation 2}
%
%
%
%
%
\end{eqnarray}
We also know from \cite{guptanewman} that the weight $5$ relators that are consequences of the $3$-Engel identity involving only $a,b,c$ are
%
\begin{eqnarray}
                [a,b,c,b,c] & = & 3[a,b,b,c,c]; \nonumber\\
\mbox{}   [a,b,c,c,b] & = & -4 [a,b,b,c,c]; \nonumber\\
\mbox{} [a,c,b,c,b] & = & 3[a,b,b,c,c]; \nonumber\\
\mbox{} [a,c,b,b,c] & = & -4[a,b,b,c,c]; \nonumber\\
\mbox{} [a,c,c,b,b] & = & [a,b,b,c,c]; \nonumber\\
\mbox{} \label{eq: relation} p[a,b,b,c,c] & = & 0.
\end{eqnarray}

Finally we consider $(a,b,c,d_{3})$. The 3-Engel relations reduce to  
\begin{eqnarray*}
    0 & = & [a,d_{3},b,c]+[a,d_{3},c,b]+[a,b,d_{3},c]+[a,c,d_{3},b]+[a,b,c,d_{3}]+[a,c,b,d_{3}]; \\
   0 & = & [d_{3},a,b,c]+[d_{3},a,c,b]+[d_{3},b,a,c]+[d_{3},b,c,a]+[d_{3},c,a,b]+[d_{3},c,b,a].
\end{eqnarray*}
From the first equation we get (using (\ref{eq: relation 1})-(\ref{eq: relation 2}))
\begin{eqnarray*}
      0 & = & \frac{1}{p}[a,[b,c],b,c]] + \frac{1}{p}[a,[b,c],c,b]+\frac{1}{p}[a,b,[b,c],c]+\frac{1}{p}[a,c,[b,c],b]+\frac{1}{p}[a,b,c,[b,c]]+\frac{1}{p}[a,c,b,[b,c]] \\
 \mbox{}      & = & \frac{1}{p}[a,b,c,b,c]-\frac{1}{p}[a,c,b,b,c]+\frac{1}{p}[a,b,c,c,b]-\frac{1}{p}[a,c,b,c,b]+\frac{1}{p}[a,b,b,c,c]-\frac{1}{p}[a,b,c,b,c] \\
 \mbox{}& + & \frac{1}{p}[a,c,b,c,b]-\frac{1}{p}[a,c,c,b,b]+\frac{1}{p}[a,b,c,b,c]-\frac{1}{p}[a,b,c,c,b]+\frac{1}{p}[a,c,b,b,c]-\frac{1}{p}[a,c,b,c,b] \\
 \mbox{}           & = & \frac{3}{p}[a,b,c,b,c]-\frac{3}{p}[a,c,b,c,b].
\end{eqnarray*}
This gives us
\begin{equation}
\frac{1}{p}[a,b,c,b,c]=\frac{1}{p}[a,c,b,c,b].
\end{equation}
%
Notice also that (\ref{eq: relation}) is compatible with and a consequence of $p^{5}x=0$, as $p[a,b,b,c,c]=p^{5}g_{1}$. It follows from this analysis that in  $K/J$ we have
that $\overline{g_{2}}$ is of order $p^{5}$ and in particular $0\not = p^{4}\bar{g_{2}}=\overline{[a,b,c,b,c]}$.
Hence $K/J$ is a $3$-Engel powerful Lie ring of nilpotency class $5$. \\ \\
We now want to construct an analogous group theory example. Unfortunately we cannot use the Mal'cev correspondence directly but guided by our Lie ring example above we can imitate the process. We start with a ${\Q}$-powered group 
$F^{*}=\langle a,b,c\rangle^{\Q}$ corresponding to our Lie algebra $F$ and we let $I^{*}$ be the normal ${\Q}$-subgroup generated by:
\begin{itemize}
    \item All commutators in $\{a,b,c\}$ of length $6$.
    \item All commutators in $\{a,b,c\}$ of length $5$ with $2$ entries in $a$.
    \item All commutators in $\{a,b,c\}$ with $3$ entries in $b$.
    \item All commutators in $\{a,b,c\}$ with $3$ entries in $c$.
    \item All commutators in $\{b,c\}$ with $2$ entries in $b$ and $2$ entries in $c$.
\end{itemize}
Let $L^{*}=F^{*}/I^{*}$ and consider the subgroup $M^{*}=\langle a^{1/p}, b^{1/p}, c^{1/p}\rangle$. Then let $K^{*}=(M^{*})^{p}$. By Lemma \ref{lem: nilpotent so powerful} we know that $[K^{*},K^{*}]\leq (K^{*})^{p}$. Notice that $K^{*}$ is a torsion-free group that
is nilpotent of class $5$ and thus a residually finite $p$-group. Pick $l$ large enough so that $[a,b,b,c,c]\not\in (K^{*})^{p^{l}}$. In particular $R^{*}=K^{*}/(K^{*})^{p^{l}}$ is nilpotent of class $5$. We now consider the verbal subgroup $J^{*}$
generated by $x^{p^{l}}$ and $[y,x,x,x]$. As $p>3$, $J$ is also generated by $x^{p^{5}}$ and the ``multi linear word''
                    $$\prod_{\sigma\in S_{3}}[x,y_{\sigma(1)},y_{\sigma(2)},y_{\sigma(3)}].$$
Observe that $R^{*}$ is generated by $a,b,c,d_{1}=[a^{1/p},b^{1/p}]^{p}, d_{2}=[a^{1/p},c^{1/p}]^{p}$ and $d_{3}=[b^{1/p},c^{1/p}]^{p}$ and $\gamma_{3}(K^{*})$. As $\gamma_{3}(K^{*})\leq Z_{3}(K^{*})$, one sees for similar reasons as in the Lie ring case, that one only needs to consider the Engel relation of multi-weight in  $\{b,c,c,d_{1}\} $, $\{ b,b,c,d_{2}\}$, $\{a,b,c,d_{3}\}$, as well as those that hold in $\langle a,b,c\rangle $. Similar calculations as for the Lie ring case show that
we get the following defining relations
for $J^{*}$ (together with $x^{p^{l}}=1$ using \cite{guptanewman})
\begin{eqnarray*}
   [a,b,b,c,c]^{1/p} & = & [a,b,c,b,c]^{2/p}; \\
\mbox{} [a,b,c,c,b]^{1/p} & = & [a,b,c,b,c]^{-3/p};  \\
   \mbox{}[a,c,c,b,b]^{1/p} & = & [a,c,b,c,b]^{2/p};  \\
               \mbox{}[a,c,b,b,c]^{1/p} & = & -[a,c,b,c,b]^{-3/p}; \\
%
%
%
%
%
 \mbox{}                [a,b,c,b,c]^{1/p} & = & [a,c,b,c,b]^{1/p}; \\
\mbox{}             [a,c,b,c,b]^{5} & = & 1.
\end{eqnarray*}
%
Thus the group $K^{*}/J^{*}$ is a powerful $3$-Engel $5$-group that is nilpotent of class $5$. \\ \\
We now turn to the case $p=2$. We let $q=16$. As the approach is very similar, we only outline the proof. For the Lie ring case, we start with the same $F$ and $I$. Now consider the Lie subring of $L$ generated by $\frac{1}{16}a, \frac{1}{16}b$ and $\frac{1}{16}c$, $M=\langle \frac{1}{16}a,\frac{1}{16}b,\frac{1}{16}c\rangle_{\Z}$, and define $K=16M\subseteq L$. This has basis $a$, $b$, $c$, $d_{1}=\frac{1}{q}[a,b]$, $d_{2}=\frac{1}{q}[a,c]$, $d_{3}=\frac{1}{q}[b,c]$, $e_{1}=\frac{1}{q^{2}}[a,b,b]$, $e_{2}=\frac{1}{q^{2}}[a,b,c]$, $e_{3}=\frac{1}{q^{2}}[a,c,b]$, $e_{4}=\frac{1}{q^{2}}[a,c,c]$, $e_{5}=\frac{1}{q^{2}}[b,c,c]$, $e_{6}=\frac{1}{q^{2}}[c,b,b]$, $f_{1}=\frac{1}{q^{3}}[a,b,b,c]$, $f_{2}=\frac{1}{q^{3}}[a,b,c,b]$, $f_{3}=\frac{1}{q^{3}}[a,c,b,b]$, $f_{4}=\frac{1}{q^{3}}[a,c,c,b]$, $f_{5}=\frac{1}{q^{3}}[a,c,b,c]$, $f_{6}=\frac{1}{q^{3}}[a,c,b,b]$, $g_{1}=\frac{1}{q^{4}}[a,b,b,c,c]$, $g_{2}=\frac{1}{q^{4}}[a,b,c,b,c]$, $g_{3}=\frac{1}{q^{4}}[a,b,c,c,b]$, $g_{4}=\frac{1}{q^{4}}[a,c,c,b,b]$, $g_{5}=\frac{1}{q^{4}}[a,c,b,c,b]$, $g_{6}=\frac{1}{q^{4}}[a,c,b,b,c]$. 

Notice that $\gamma_{5}(K)\cap 4q^{4}K=4\gamma_{5}(K)$ and thus $R=K/4q^{4}K$ is nilpotent of class $5$ and also has dimension $24$ as an abelian $2$-group.
Then
$$
[K,K]=[qM,qM]=q^{2}[M,M]\leq 4\cdot q M=4K,
$$
so that $R=K/4q^{4}K$ is a powerful $2$-Lie ring that also has rank $24$ as an abelian $2$-group. Now consider the verbal ideal $J$ of $K$ generated by the words $q^{5}x$ and $[y,x,x,x]$.
Things are slightly more complicated here in that the ideal $J$ is not multi-graded.
It is though graded with respect to total weight of the generators. More precisely the identity $[y,x,x,x]=0$ is equivalent to saying that the following hold for all $x,y,z,x_{1},x_{2},x_{3}$ from the generating set above. 
\begin{eqnarray}
    [y,x,x,x] & = & 0; \nonumber \\
\mbox{}[y,z,x,x]+[y,x,z,x]+[y,x,x,z]+[y,x,z,z]+[y,z,x,z]+[y,z,z,x] & = & 0; \label{eq: relation 3}\\
\mbox{}\sum_{\sigma\in S_{3}}[y,x_{\sigma(1)},x_{\sigma(2)},x_{\sigma(3)}] & = & 0.\label{eq: relation 4}
\end{eqnarray}
%
%
 Note that again $e_{1},\ldots ,e_{6},f_{1},\ldots ,f_{6}, g_{1},\ldots ,g_{6}$ are in the third center and that the only homogeneous components that are not consequences of the $3$-Engel identity involving only $a,b,c$ are those using (\ref{eq: relation 3}) for 
$(b,c,d_{1})$, $(b,c,d_{2})$, and (\ref{eq: relation 4}) for $(a,b,c,d_{3})$. We know already from \cite{guptanewman} that without those identities the algebra would be nilpotent of class $5$. As in the $p=5$ case these additional identities are compatible with this and
we can conclude that the class remains $5$. The group situation works then similarly as before. 
\end{proof}

Finally, we summarize the results obtained in the last two sections.

\begin{proof}[Proof of Theorem \ref{thm: non-metabelian}]
Parts (i) and (iii) follow from Theorem \ref{thm: r=3, p neq2}, and these bounds are best possible since every group of nilpotency class $3$ is always $3$-Engel. Part (ii) follows from Theorem \ref{thm: gamma5^20=1}, and this upper bound is best possible by Examples \ref{ex: p=3} and \ref{ex: p>3}. Part (iv) follows from Theorem \ref{thm: r=3, p=2}, and this upper bound is best possible by Example \ref{ex: p=2,r=3,powerful}. Part (v) follows from Theorem \ref{thm: r=4,5}, and this upper bound is best possible by Examples \ref{ex: p=2,r=3,powerful} and \ref{ex: p>3}. Part (vi) follows from Theorem \ref{thm: gamma6=1} and Theorem \ref{thm: r>5}.
\end{proof}

\bibliography{references}

\end{document}